\documentclass[12pt,reqno]{article}

\usepackage[usenames]{color}
\usepackage{amssymb}
\usepackage{graphicx}
\usepackage{amscd}

\usepackage{amsthm}
\newtheorem{theorem}{Theorem}

\newtheorem{proposition}[theorem]{Proposition}
\newtheorem{corollary}[theorem]{Corollary}
\newtheorem{conjecture}[theorem]{Conjecture}

\theoremstyle{definition}

\newtheorem{example}[theorem]{Example}

\usepackage[colorlinks=true,
linkcolor=webgreen, filecolor=webbrown,
citecolor=webgreen]{hyperref}

\definecolor{webgreen}{rgb}{0,.5,0}
\definecolor{webbrown}{rgb}{.6,0,0}

\usepackage{color}

\usepackage{float}

\usepackage{graphics,amsmath,amssymb}
\usepackage{amsfonts}
\usepackage{latexsym}
\usepackage{epsf}

\newcommand{\seqnum}[1]{\href{http://www.research.att.com/cgi-bin/access.cgi/as/~njas/sequences/eisA.cgi?Anum=#1}{\underline{#1}}}

\begin{document}

\begin{center}
\vskip 1cm{\LARGE\bf  Invariant number triangles, eigentriangles and Somos-$4$ sequences}  \vskip 1cm
\large
Paul Barry\\
School of Science\\
Waterford Institute of Technology\\
Ireland\\
\href{mailto:pbarry@wit.ie}{\tt pbarry@wit.ie} \\
\end{center}
\vskip .2 in

\begin{abstract} Using the language of Riordan arrays, we look at two related iterative processes on matrices and determine
which matrices are invariant under these processes. In a special case, the invariant sequences that arise are conjectured to have Hankel transforms that obey Somos-$4$ recurrences. A notion of eigentriangle for a number triangle emerges and examples are given, including a construction of the Takeuchi numbers.
\end{abstract}

\section{Introduction}
In this note, we shall define  transformations on invertible lower-triangular matrices involving the down-shifting of elements and taking an inverse. The invariant matrices for these transformations turn out to be simple Riordan arrays \cite{SGWW}, with generating functions easily described by continued fractions \cite{CF, Wall}. These matrices have close links to the Catalan numbers $C_n=\frac{1}{n+1}\binom{2n}{n}$. In the case of a particular two-parameter transformation, special sequences defined by this process appear to have Hankel transforms that satisfy Somos-$4$ type recurrences \cite{Somos}. Again using Riordan arrays we can characterize these sequences.

We recall that the \emph{Riordan group} \cite{SGWW, Spru}, is a set of
infinite lower-triangular integer matrices, where each matrix is
defined by a pair of generating functions
$g(x)=1+g_1x+g_2x^2+\ldots$ and $f(x)=f_1x+f_2x^2+\ldots$ where
$f_1\ne 0$ \cite{Spru}. The associated matrix is the matrix whose
$i$-th column is generated by $g(x)f(x)^i$ (the first column being
indexed by 0). The matrix corresponding to the pair $f, g$ is
denoted by $(g, f)$. The group law is then given
by
\begin{displaymath} (g, f)\cdot (h, l)=(g(h\circ f), l\circ
f).\end{displaymath} The identity for this law is $I=(1,x)$ and the
inverse of $(g, f)$ is $(g, f)^{-1}=(1/(g\circ \bar{f}), \bar{f})$
where $\bar{f}$ is the compositional inverse of $f$. This is also called the (series) reversion of $f$.
A Riordan array of the form $(g(x),x)$, where $g(x)$ is the
generating function of the sequence $a_n$, is called the
\emph{sequence array} of the sequence $a_n$. Its general term is
$a_{n-k}$, or more accurately $[k \le n]a_{n-k}$ (where $[P]$ is the Iverson bracket \cite{Concrete},
defined by $[\mathcal{P}]=1$ if the proposition $\mathcal{P}$
is true, and
$[\mathcal{P}]=0$ if $\mathcal{P}$ is false). Such arrays are also called \emph{Appell} arrays as they form the elements of the so called
Appell subgroup.
\newline\newline If $\mathbf{M}$ is the matrix $(g,f)$, and
$\mathbf{a}=(a_0,a_1,\ldots)'$ is an integer sequence with ordinary
generating function $\cal{A}$ $(x)$, then the sequence
$\mathbf{M}\mathbf{a}$ has ordinary generating function
$g(x)$$\cal{A}$$(f(x))$. The (infinite) matrix $(g,f)$ can thus be considered to act on the ring of
integer sequences $\mathbf{Z}^\mathbf{N}$ by multiplication, where a sequence is regarded as a
(infinite) column vector. We can extend this action to the ring of power series
$\mathbf{Z}[[x]]$ by
$$(g,f):\cal{A}(\mathnormal{x}) \longrightarrow \mathnormal{(g,f)}\cdot
\cal{A}\mathnormal{(x)=g(x)}\cal{A}\mathnormal{(f(x))}.$$
\begin{example} The binomial matrix $\mathbf{B}$ is the element
$(\frac{1}{1-x},\frac{x}{1-x})$ of the Riordan group. It has general
element $\binom{n}{k}$. More generally, $\mathbf{B}^m$ is the
element $(\frac{1}{1-m x},\frac{x}{1-mx})$ of the Riordan group,
with general term $\binom{n}{k}m^{n-k}$. It is easy to show that the
inverse $\mathbf{B}^{-m}$ of $\mathbf{B}^m$ is given by
$(\frac{1}{1+mx},\frac{x}{1+mx})$.
\end{example}
\noindent In the sequel, we shall assume that all matrices and sequences are integer valued.
\section{The $(a,b)$-Process}
We start by defining an operation on
lower-triangular matrices which have $1$'s on the diagonal.
Thus let $M$ be of the form
\begin{equation}\label{M} M=\left(\begin{array}{ccccccc} 1 & 0 &
0
& 0 & 0 & 0 & \ldots \\m_{2,1} & 1 & 0 & 0 & 0 & 0 & \ldots \\ m_{3,1} & m_{3,2}
& 1 & 0 & 0 &
0 & \ldots \\ m_{4,1} & m_{4,2} & m_{4,3} & 1 & 0 & 0 & \ldots \\ m_{5,1} & m_{5,2} & m_{5,3}
& m_{5,4} & 1 & 0 & \ldots \\m_{6,1} & m_{6,2} & m_{6,3} & m_{6,4} & m_{6,5} & 1
&\ldots\\
\vdots &
\vdots & \vdots & \vdots & \vdots & \vdots &
\ddots\end{array}\right).\end{equation}
Now form the matrix
\begin{equation}\label{M1} \tilde{M}(a,b)=\left(\begin{array}{ccccccc} 1 & 0 &
0
& 0 & 0 & 0 & \ldots \\-a & 1 & 0 & 0 & 0 & 0 & \ldots \\ -b & -a
& 1 & 0 & 0 &
0 & \ldots \\ -m_{2,1} & -b & -a & 1 & 0 & 0 & \ldots \\ -m_{3,1} & -m_{3,2} & -b
& -a & 1 & 0 & \ldots \\-m_{4,1} & -m_{4,2} & -m_{4,3} & -b & -a & 1
&\ldots\\
\vdots &
\vdots & \vdots & \vdots & \vdots & \vdots &
\ddots\end{array}\right).\end{equation}
Then we take the inverse $\tilde{M}(a,b)^{-1}$ of this matrix. Let us call this process the $(a,b)$-process.
We have the following proposition.
\begin{proposition} Let $f(x)$ be the power series defined by
\begin{equation} \label{eq}f(x)=\frac{1}{1-ax-(b-1)x^2-x^2f(x)}.\end{equation}
Then the Riordan array $$(f(x),x)$$ is invariant under the $(a,b)$-operation.
\end{proposition}
\begin{proof}
By equation (\ref{eq}), we see that $f(x)=\sum_{i=0}^{\infty} {a_i x^i}$ where $a_0=1$. Then
$$x^2 f(x)=x^2 a_0+x^3 \sum_{i=0}a_{i+1}x^i=x^2+x^3 \sum_{i=0}a_{i+1}x^i.$$
We obtain
$$1-ax-(b-1)x^2-x^2 f(x)=1-ax-bx^2+x^2-x^2-x^3\sum_{i=0}a_{i+1}x^i=1-ax-bx^2-x^3\sum_{i=0}a_{i+1}x^i.$$
Thus we wish to prove that
$$(f(x),x)=(1-ax-(b-1)x^2-x^2 f(x),x)^{-1},$$ or equivalently that
$$(f(x),x)^{-1}=(1-ax-(b-1)x^2-x^2 f(x),x).$$
Now $$(f(x),x)^{-1}=\left(\frac{1}{f(x)},x\right)$$ and hence we wish to establish that
$$\frac{1}{f(x)}=1-ax-(b-1)x^2-x^2 f(x).$$ But this follows immediately from the definition of $f$.
\end{proof}

Let $a_n$ denote the $n$-th element of the first column of $(f(x),x)$. Then the $(n,k)$-th element of
$(f(x),x)$ is given by
$$[k \le n] a_{n-k}.$$ Thus we need only a knowledge of $a_n$ to describe all elements of the matrix.
\begin{proposition} Let $$g(x)=\frac{1}{1-ax-bx^2-x^2 g(x)}. $$
Then
$$[x^n]g(x)=\sum_{k=0}^{\lfloor \frac{n}{2} \rfloor}\binom{n-k}{k}b^k \sum_{j=0}^{n-2k}\binom{n-2k}{j}a^{n-2k-j}C_{\frac{j}{2}}\frac{1+(-1)^j}{2},$$
where $C_n=\frac{1}{n+1}\binom{2n}{n}$ is the $n$-th Catalan number \seqnum{A000108}.
\end{proposition}
\begin{proof}
Solving the equation
$$g(x)=\frac{1}{1-ax-bx^2-x^2g(x)}$$ gives us
$$g(x)=g_{a,b}(x)=\frac{1-ax-bx^2-\sqrt{1-2ax+(a^2-2b-4)x^2+2abx^3+b^2x^4}}{2x^2}.$$ With this value, we then have the Riordan array factorization
\begin{eqnarray*}(g_{a,b}(x),x)&=&\left(\frac{1}{1-ax-bx^2},\frac{x}{1-ax-bx^2}\right)\cdot \left(c(x^2), \frac{g_{a,b}(x)}{x}\right)\\
&=&\left(\frac{1}{1-bx^2},\frac{x}{1-bx^2}\right)\cdot \left(\frac{1}{1-ax},\frac{x}{1-ax}\right) \cdot \left(c(x^2),\frac{g_{a,b}(x)}{x}\right),\end{eqnarray*}
where $$c(x)=\frac{1-\sqrt{1-4x}}{2x}$$ is the g.f. of the Catalan numbers, and $c(x^2)$ is the g.f. of the aerated
Catalan numbers $1,0,1,0,2,0,5,0,\ldots$.  Thus
$$[x^n]g(x)=[x^n]\left(\frac{1}{1-bx^2},\frac{x}{1-bx^2}\right)\cdot \left(\frac{1}{1-ax},\frac{x}{1-ax}\right) \cdot c(x^2).$$ The result follows from this.

\end{proof}

\begin{corollary}
$$a_n=\sum_{k=0}^{\lfloor \frac{n}{2} \rfloor}\binom{n-k}{k}(b-1)^k \sum_{j=0}^{n-2k}\binom{n-2k}{j}a^{n-2k-j}C_{\frac{j}{2}}\frac{1+(-1)^j}{2}.$$

\end{corollary}
We note that if we start with any matrix of the form (\ref{M}), and iterate the $(a,b)$-process on it, then the limit matrix is $(f(x),x)$. Thus the element of the Appell subgroup of the Riordan group $(f(x),x)$ where
$$f(x)=\cfrac{1}{1-ax-(b-1)x^2-
\cfrac{x^2}{1-ax-(b-1)x^2-
\cfrac{x^2}{1-\cdots}}},$$ is a ``universal element'' for the $(a,b)$-process.
\section{A Somos-$4$ conjecture}
We have the following Somos-$4$ conjecture.
\begin{conjecture} The Hankel transform of the sequence $a_n$ is a $(a^2, b^2-a^2)$ Somos-$4$ sequence.
\end{conjecture}

By this we mean that the sequence $h_n$ of Hankel determinants
$$h_n=| a_{i+j}|_{0 \le i,j \le n}$$ satisfies an $(\alpha, \beta)$ Somos-$4$ relation
$$ h_n=\frac{\alpha h_{n-1} h_{n-3} + \beta h_{n-2}^2}{h_{n-4}}, \quad n >3,$$ where
$\alpha=a^2$ and $\beta=b^2-a^2$.

Equivalently  the Hankel transform of the sequence with general term
$$\sum_{k=0}^{\lfloor \frac{n}{2} \rfloor}\binom{n-k}{k}b^k \sum_{j=0}^{n-2k}\binom{n-2k}{j}a^{n-2k-j}C_{\frac{j}{2}}\frac{1+(-1)^j}{2}$$ is (conjectured to be) a $(a^2, (b+1)^2-a^2)$ Somos-$4$ sequence.
\begin{example}
We let $a=b=1$. Then $a_n$ is the sequence \seqnum{A128720}
$$1, 1, 3, 6, 16, 40, 109, 297, 836, 2377, 6869\ldots$$ which counts the number of skew Dyck paths of semi-length $n$ with no $UUU$'s. The Hankel transform of this sequence is the $(1,3)$ Somos-$4$ sequence \seqnum{A174168} which begins
$$1, 2, 5, 17, 109, 706, 9529, 149057, 3464585, 141172802, 5987285341,\ldots.$$
\end{example}
\begin{example} We take $a=1$, $b=2$ to get the sequence
\seqnum{A174171} which begins
$$1, 1, 4, 8, 25, 65, 197, 571, 1753, 5351, 16746\ldots,$$ with $(1,8)$ Somos-$4$ Hankel transform
$$1, 3, 11, 83, 1217, 22833, 1249441, 68570323, 11548470571, 2279343327171,\ldots.$$
This is \seqnum{A097495}, or the even-indexed terms of the Somos-$5$ sequence.
\end{example}
\begin{example} We let $a=2$, and $b=-1$. Then $a_n$ is the sequence \seqnum{A187256} which begins
$$1, 2, 4, 10, 28, 82, 248, 770, 2440, 7858, 25644, \ldots.$$
This sequence counts peakless Motzkin paths where the level steps come in two colours (Deutsch). The Hankel transform of this sequence is the Somos-$4$ variant \seqnum{A162547} that begins
$$1, 0, -4, -16, -64, 0, 4096, 65536, 1048576, 0, -1073741824,\ldots.$$
\end{example}

\section{The ``$(a)$-process'' and Narayana numbers}
We now look at the simpler ``$(a)$-process'', whereby we send the matrix

\begin{equation}\label{M1} M=\left(\begin{array}{ccccccc} 1 & 0 &
0
& 0 & 0 & 0 & \ldots \\m_{2,1} & 1 & 0 & 0 & 0 & 0 & \ldots \\ m_{3,1} & m_{3,2}
& 1 & 0 & 0 &
0 & \ldots \\ m_{4,1} & m_{4,2} & m_{4,3} & 1 & 0 & 0 & \ldots \\ m_{5,1} & m_{5,2} & m_{5,3}
& m_{5,4} & 1 & 0 & \ldots \\m_{6,1} & m_{6,2} & m_{6,3} & m_{6,4} & m_{6,5} & 1
&\ldots\\
\vdots &
\vdots & \vdots & \vdots & \vdots & \vdots &
\ddots\end{array}\right)\end{equation} to the matrix

\begin{displaymath}\tilde{M}_a=\left(\begin{array}{ccccccc} 1 & 0 &
0
& 0 & 0 & 0 & \ldots \\-a & 1 & 0 & 0 & 0 & 0 & \ldots \\ -m_{2,1} & -a
& 1 & 0 & 0 &
0 & \ldots \\ -m_{3,1} & -m_{3,2} & -a & 1 & 0 & 0 & \ldots \\ -m_{4,1} & -m_{4,2} & -m_{4,3}
& -a & 1 & 0 & \ldots \\-m_{5,1} & -m_{5,2} & -m_{5,3} & -m_{5,4} & -a & 1
& \ldots\\
\vdots &
\vdots & \vdots & \vdots & \vdots & \vdots &
\ddots\end{array}\right),\end{displaymath} and then take the inverse to obtain
$\tilde{M}_a^{-1}$.
We have the following result.
\begin{proposition} Let $f(x)$ be the power series defined by
$$f(x)=\frac{1}{1-(a-1)x-xf(x)}.$$ Then the Riordan array
$$(f(x),x)$$ is invariant under the $(a)$-process.
\end{proposition}
\begin{proof} We wish to show that
$$(f(x),x)=(1-(a-1)x-xf(x),x)^{-1},$$ or equivalently that
$$(f(x),x)^{-1}=\left(\frac{1}{f(x)},x\right)=(1-(a-1)x-xf,x).$$
But this follows immediately since by definition
$$f(x)=\frac{1}{1-(a-1)x-xf(x)}.$$
\end{proof}
\noindent We now remark that the continued fraction
$$f(x)=\cfrac{1}{1-(a-1)x-
\cfrac{x}{1-(a-1)x-
\cfrac{x}{1-\cdots}}}$$ is the generating function of the
Narayana polynomials $\mathcal{N}_n(a)=\sum_{k=0}^n N_{n,k}a^k$ \cite{Narayana_Gen, Narayana_Mimo, Sulanke}
where the matrix $(N_{n,k})$ is the matrix of Narayana numbers \seqnum{A090181}
\begin{displaymath}\left(\begin{array}{ccccccc} 1 & 0 & 0 & 0
&
0 & 0 & \ldots \\0 &
1 & 0 & 0 & 0 & 0 & \ldots \\ 0 & 1 & 1 & 0 & 0 & 0 & \ldots
\\
0 & 1 & 3 & 1 & 0 & 0 & \ldots \\ 0 & 1 & 6 & 6 & 1 & 0 &
\ldots \\0 & 1 &
10 & 20 & 10 & 1 &\ldots\\ \vdots & \vdots & \vdots & \vdots &
\vdots & \vdots & \ddots\end{array}\right).\end{displaymath}
Hence the terms of the first column of $(f(x),x)$ are precisely the
Narayana polynomials in $a$:
$$a_n=\mathcal{N}_n(a)=\sum_{k=0}^n N_{n,k}a^k.$$
In particular, for $a=1$, we get
$$a_n=C_n,$$ the Catalan numbers.

As before, we note that if we start from an arbitrary matrix of the form Eq. (\ref{M1}), and iterate the
$(a)$-process, then the limit matrix is $(f(x),x)$. In particular, if $a=1$, the limit matrix is the Catalan numbers sequence array $(C_{n-k})$:
\begin{displaymath}\left(\begin{array}{ccccccc} 1 & 0 &
0
& 0 & 0 & 0 & \ldots \\1 & 1 & 0 & 0 & 0 & 0 & \ldots \\ 2 & 1
& 1 & 0 & 0 &
0 & \ldots \\ 5 & 2 & 1 & 1 & 0 & 0 & \ldots \\ 14 & 5 & 2
& 1 & 1 & 0 & \ldots \\42 & 14 & 5 & 2 & 1 & 1
&\ldots\\
\vdots &
\vdots & \vdots & \vdots & \vdots & \vdots &
\ddots\end{array}\right).\end{displaymath}
\noindent This is the Riordan array $(c(x),x)$.

By solving the equation $$f(x)=\frac{1}{1-(a-1)x-xf(x)}$$ we see that
$$(f(x),x)=\left(\frac{1-(a-1)x-\sqrt{1-2(a+1)x+(a-1)^2 x^2}}{2x},x\right),$$ which by the above is the matrix
with $(n,k)$-th term
$$[k \le n] \mathcal{N}_{n-k}(a).$$
\section{Eigentriangles}
We also have the following result.
\begin{proposition}

Let $M$ be a matrix as in Eq. (\ref{M}). Then $\tilde{M}_1^{-1}$ is an eigentriangle of
$M$.
\end{proposition}
\noindent By this we mean that if
\begin{equation}\label{M2} \tilde{M}_1^{-1}=\left(\begin{array}{ccccccc} 1 & 0 &
0
& 0 & 0 & 0 & \ldots \\1 & 1 & 0 & 0 & 0 & 0 & \ldots \\ r_{3,1} & 1
& 1 & 0 & 0 &
0 & \ldots \\ r_{4,1} & r_{4,2} & 1 & 1 & 0 & 0 & \ldots \\ r_{5,1} & r_{5,2} & r_{5,3}
& 1 & 1 & 0 & \ldots \\r_{6,1} & r_{6,2} & r_{6,3} & r_{6,4} & 1 & 1
&\ldots\\
\vdots &
\vdots & \vdots & \vdots & \vdots & \vdots &
\ddots\end{array}\right)\end{equation}
then
\begin{displaymath}M \tilde{M}_1^{-1}=\left(\begin{array}{ccccccc} 1 & 0 & 0
& 0 & 0 & 0 & \ldots \\r_{3,1} & 1 & 0 & 0 & 0 & 0 & \ldots \\
r_{4,1} & r_{4,2} & 1 & 0 & 0 & 0 & \ldots \\ r_{5,1} & r_{5,2} & r_{5,3} & 1
& 0 & 0 & \ldots \\ r_{6,1} & r_{6,2} & r_{6,3} & r_{6,4} & 1 & 0 & \ldots
\\r_{7,1} & r_{7,2} & r_{7,3} & r_{7,4} & r_{7,5} & 1 &\ldots\\ \vdots & \vdots & \vdots & \vdots & \vdots
& \vdots & \ddots\end{array}\right).\end{displaymath}
\noindent Note that the first column of $\tilde{M}_1^{-1}$ is then an \emph{eigensequence} of $M$.
\begin{proof} We have $$\tilde{M} \tilde{M}_1^{-1}=I$$ and hence
$$-\sum_{j=1}^{k-1} m_{k-1,j}r_{j,l}+r_{k,l}=0 \quad \text{for} \quad k \ne l.$$
Then for $k \ne l$, we have
$$r_{k,l}=\sum_{j=0}^{k-1} m_{k-1,j}r_{j,l}.$$
Thus the $(k-1,l)$-th element of $M \tilde{M}_1^{-1}$ is $r_{k,l}$.
\end{proof}
\begin{example} The eigentriangle of the binomial matrix $(\binom{n}{k})$ is given by
\begin{displaymath}E=\left(\begin{array}{ccccccc} 1 & 0 & 0 & 0 & 0 & 0 & \ldots \\1 & 1 & 0 & 0 & 0 & 0 & \ldots \\ 2 & 1 & 1 & 0 & 0 & 0 & \ldots \\ 5 & 3 & 1 & 1 & 0 & 0 & \ldots \\ 15 & 9 & 4 & 1 & 1 & 0 & \ldots \\52 & 31 & 14 & 5 & 1 & 1 &\ldots\\ \vdots & \vdots & \vdots & \vdots & \vdots & \vdots & \ddots\end{array}\right),\end{displaymath} where the first column entries are the Bell numbers.
We note in passing that the production matrix \cite{ProdMat} of the matrix $E$ is equal to
\begin{displaymath}\left(\begin{array}{ccccccc} 1 & 1 & 0
& 0 & 0 & 0 & \ldots \\1 & 0 & 1 & 0 & 0 & 0 & \ldots \\
2 & 1 & 0 & 1 & 0 & 0 & \ldots \\ 5 & 3 & 1 & 0
& 1 & 0 & \ldots \\ 15 & 9 & 4 & 1 & 0 & 1 & \ldots
\\52 & 31 & 14 & 5 & 1 & 0 &\ldots\\ \vdots & \vdots & \vdots & \vdots & \vdots
& \vdots & \ddots\end{array}\right).\end{displaymath}
In this case, we have
$$a_n=\sum_{k=0}^{n-1} \binom{n-1}{k}a_k, \quad n>0, \quad a_0=1,$$ or $$a_n= Bell(n),$$ the Bell numbers \seqnum{A000110}.
\end{example}
\begin{example}
The eigentriangle of the skew binomial matrix $(\binom{k}{n-k})$ is given by
\begin{displaymath}E=\left(\begin{array}{ccccccc} 1 & 0 & 0 & 0 & 0 & 0 & \ldots \\1 & 1 & 0 & 0 & 0 & 0 & \ldots \\ 1 & 1 & 1 & 0 & 0 & 0 & \ldots \\ 2 & 2 & 1 & 1 & 0 & 0 & \ldots \\ 4 & 4 & 3 & 1 & 1 & 0 & \ldots \\11 & 11 & 7 & 4 & 1 & 1 &\ldots\\ \vdots & \vdots & \vdots & \vdots & \vdots & \vdots & \ddots\end{array}\right),\end{displaymath}
where the first column
$$1, 1, 1, 2, 4, 11, 33, 114, 438, 1845, 8458,\ldots$$ or \seqnum{A127782} is thus an eigensequence of
$(\binom{k}{n-k})$ (remark by Gary W. Adamson).
We have
$$a_n=\sum_{k=0}^{n-1} \binom{k}{n-k-1}a_k, \quad n>0, \quad a_0=1.$$
\end{example}
\begin{example}
The eigentriangle of the sequence array for the Motzkin numbers $M_n$ (i.e., the matrix with
$(n,k)$-th term $[k \le n] M_{n-k}$ where $M_n=\sum_{k=0}^{\lfloor \frac{n}{2} \rfloor} \binom{n}{2k}C_k$) is the sequence array
for the sequence \seqnum{A005773} of directed animals $A_n$ of size $n$.
Thus
$$A_n=\sum_{k=0}^{n-1} M_{n-k-1}A_k.$$
\end{example}
\noindent We can characterize the eigentriangle $E=(E(n,k))$ corresponding to a matrix $A=(A(n,k))$ as follows.
We define
\begin{equation}\label{Eigen}\tilde{E}(n,j)=\sum_{k=0}^{n-1} A(n-1+j,k+j)\tilde{E}(k,j), \quad \text{with} \quad \tilde{E}(0,j)=1.\end{equation}
Then $$E(n,k)=[k \le n] \tilde{E}(n-k,k).$$
\section{The Takeuchi numbers}
The Takeuchi numbers $t_n$ \seqnum{A000651} are an example of a sequence that can be defined with the aid of the eigentriangle of the Catalan triangle $(c(x),xc(x))$ \seqnum{A033184}. We let $T(x)$ be the generating function of the Takeuchi numbers.
Our point of departure is $(4)$ in \cite{Prellberg}:
$$T(x)=\frac{c(x)-1}{1-x}+\frac{x(2-c(x))}{\sqrt{1-4x}}T(xc(x)).$$
We now note that
$$\frac{(2-c(x))}{\sqrt{1-4x}}=c(x),$$ so that \cite{Prellberg}$(4)$ becomes
$$T(x)=\frac{c(x)-1}{1-x}+xc(x)T(xc(x)).$$ In terms of Riordan arrays, we may write this as
$$((1,x)-(xc(x),xc(x))).T(x)=\frac{c(x)-1}{1-x}.$$
Now while the matrix  $$(1,x)-(xc(x),xc(x))$$ is not a Riordan array, it is a special type of
invertible matrix. The theory of eigentriangles tells us that its inverse is the eigentriangle of the
Catalan matrix $$(c(x),xc(x)).$$ This eigentriangle begins
\begin{displaymath}\mathbf{E}=\left(\begin{array}{ccccccc}
1 & 0 & 0 & 0 & 0 & 0 & \ldots \\1 &
1 & 0 & 0 & 0 & 0 & \ldots \\ 2 & 1 & 1 & 0 & 0 & 0 & \ldots
\\
6 & 3 & 1 & 1 & 0 & 0 & \ldots \\ 22 & 11 & 4 & 1 & 1 & 0 &
\ldots \\92 & 46 &
17 & 5 & 1 & 1 &\ldots\\ \vdots & \vdots & \vdots & \vdots &
\vdots & \vdots & \ddots\end{array}\right).\end{displaymath}
We then have
\begin{displaymath}\begin{split}
\left(\begin{array}{ccccccc}
1 & 0 & 0 & 0 & 0 & 0 & \ldots \\1 &
1 & 0 & 0 & 0 & 0 & \ldots \\ 2 & 2 & 1 & 0 & 0 & 0 & \ldots
\\
5 & 5 & 3 & 1 & 0 & 0 & \ldots \\ 14 & 14 & 9 & 4 & 1 & 0 &
\ldots \\42 & 42 &
28 & 14 & 5 & 1 &\ldots\\ \vdots & \vdots & \vdots & \vdots &
\vdots & \vdots & \ddots\end{array}\right)\left(\begin{array}{ccccccc}
1 & 0 & 0 & 0 & 0 & 0 & \ldots \\1 &
1 & 0 & 0 & 0 & 0 & \ldots \\ 2 & 1 & 1 & 0 & 0 & 0 & \ldots
\\
6 & 3 & 1 & 1 & 0 & 0 & \ldots \\ 22 & 11 & 4 & 1 & 1 & 0 &
\ldots \\92 & 46 &
17 & 5 & 1 & 1 &\ldots\\ \vdots & \vdots & \vdots & \vdots &
\vdots & \vdots & \ddots\end{array}\right)=\\\left(\begin{array}{ccccccc}
1 & 0 & 0 & 0 & 0 & 0 & \ldots \\2 &
1 & 0 & 0 & 0 & 0 & \ldots \\ 6 & 3 & 1 & 0 & 0 & 0 & \ldots
\\
22 & 11 & 4 & 1 & 0 & 0 & \ldots \\ 92 & 46 & 17 & 5 & 1 & 0 &
\ldots \\426 & 213 &
79 & 24 & 6 & 1 &\ldots\\ \vdots & \vdots & \vdots & \vdots &
\vdots & \vdots & \ddots\end{array}\right).\end{split}\end{displaymath}
The sequence with g.f. $\frac{c(x)-1}{1-x}$ is the sequence \seqnum{A014138} with general term
$$\sum_{k=0}^{n-1}C_{k+1},$$ and thus the Takeuchi numbers are the image of this sequence by $\mathbf{E}$.
Now in this case $A$ of Eq. (\ref{Eigen}) is the matrix $(c(x),xc(x))$ with $(n,k)$-th term
$$A(n,k)=\binom{2n-k}{n-k}\frac{k+1}{n+1}.$$
Thus we get
$$\tilde{E}(n,j)=\sum_{k=0}^{n-1} \binom{2(n-1)+j-k}{n-1-k}\frac{k+j+1}{n+j}\tilde{E}(k,j), \quad \text{with} \quad \tilde{E}(0,j)=1,$$
and so $$t_n= \sum_{k=0}^n \tilde{E}(n-k,k)\sum_{j=0}^{k-1} C_{j+1}.$$

\noindent We note that the first column of $\mathbf{E}$ is essentially \seqnum{A091768}.
\section{Acknowledgements}
There are many examples of eigensequences in \cite{SL1}, many of which are contributed by Paul D. Hanna or Gary W. Adamson.
One can find a different but related notion of eigentriangle therein (see \seqnum{A144218}, for example).
An alternative iterative construction of eigensequences is given, for instance, in \seqnum{A168259}.
The ``$(1)$-process'' and the $(1,1)$-process are looked at  in The Mobius function Blog of Mats Granvik \cite{Granvik}. Examples of eigentriangles as defined here are \seqnum{A172380}, \seqnum{A181644},\seqnum{A181651}, \seqnum{A181654}, \seqnum{A186020}, \seqnum{A186023}, \seqnum{A172380}.

\bigskip
\hrule
\bigskip
\noindent 2010 {\it Mathematics Subject Classification}: Primary
15B36; Secondary 11B37, 11B83, 11C20, 15B05

\noindent \emph{Keywords:} Riordan array, eigentriangle, eigensequence, Narayana numbers, Catalan numbers, Somos sequence, Hankel transform, Takeuchi number.

\bigskip
\hrule
\bigskip
\noindent Concerns sequences
\seqnum{A000108},
\seqnum{A000110},
\seqnum{A000651},
\seqnum{A014138},
\seqnum{A033184},
\seqnum{A090181},
\seqnum{A091768},
\seqnum{A097495},
\seqnum{A127782},
\seqnum{A128720},
\seqnum{A144218},
\seqnum{A162547},
\seqnum{A168259},
\seqnum{A172380},
\seqnum{A174168},
\seqnum{A174171},
\seqnum{A181644},
\seqnum{A181651},
\seqnum{A181654},
\seqnum{A186020},
\seqnum{A186023},
\seqnum{A187256}

\end{document}